\documentclass [12pt]{amsart}
\usepackage{times}
\usepackage{amssymb,latexsym,amsmath,amsfonts, eucal}
\usepackage[usenames,dvipsnames,svgnames,table]{xcolor}

\newcommand\hlight[1]{\tikz[overlay, remember picture,baseline=-\the\dimexpr\fontdimen22\textfont2\relax]\node[rectangle,fill=white!50,rounded corners,fill opacity = 0.2,draw,thick,text opacity =1] {$#1$};}

\usepackage[hidelinks]{hyperref}

\newtheorem{theorem}{Theorem}[section]
\newtheorem{lemma}{Lemma}[section]
\newtheorem{remark}{Remark}[section]
\newtheorem{proposition}{Proposition}[section]
\newtheorem{definition}{Definition}[section]
\newtheorem{example}{Example}[section]

\newtheorem{corollary}{Corollary}[section]
\numberwithin{equation}{section}
\usepackage{tikz}
\usepackage{float}
\usepackage{booktabs}
\usepackage{pgfplots}
\usepgfplotslibrary{fillbetween}
\pgfplotsset{compat=1.18}
\usetikzlibrary{shapes,arrows,positioning}
\usetikzlibrary{automata,arrows,positioning,calc}
\usepackage{enumitem}
\usepackage{subcaption}
\graphicspath{ {./images/} }

\title[ ]{Assouad type dimensions of generalized affine fractal interpolation functions and their applications}
\author[Shah]{Aaryan Dharmesh Shah}
\address{Aaryan Dharmesh Shah, Department of Mathematics, National Institute of Technology Rourkela, Rourkela-769008, India}
\email{aaryanshah1911121@gmail.com}
\author[Jha]{Sangita Jha}
\address{Sangita Jha, Department of Mathematics, National Institute of Technology Rourkela, Rourkela-769008, India}
\email{jhasa@nitrkl.ac.in, \textbf{Corresponding Author-Sangita Jha}} 



\begin{document}
\begin{abstract}
In this article, we investigate the Assouad spectrum and Assouad dimension of
graphs of fractal functions generated by generalized affine iterated function
systems. We establish upper and lower bounds for the Assouad spectrum in terms
of the scaling functions and the underlying partition of the generalized
affine construction. If the scaling function is Lipschitz
continuous and the partition is uniform, we obtain an explicit expression
for the Assouad spectrum of the associated graph. These results provide a connection between the
parameters defining the generalized affine fractal function and the local multiscale geometry of its graph.

As applications, we consider two classic examples of generalized affine fractal functions, namely the Weierstrass and Takagi functions. For the
classical Weierstrass function $W$, whose graph  $\Gamma_W$ has the box dimension
$2+\log_N\lambda$, we obtain
\[
\dim_A^\theta(\Gamma_W)
\leq
\frac{2+\log_N\lambda-\theta}{1-\theta},
\qquad
\theta\in
\left(0,\log_N\frac{1}{\lambda}\right).
\]
and if $\lambda^2 N <1$, we get
\[
\dim_A(\Gamma_W)\geq 1 +\log_N\left(\frac{1}{\lambda}\right).
\]
For the classical Takagi function $T$ with graph $\Gamma_T$, we show that
\[
\dim_A^\theta(\Gamma_T)=1, \theta\in(0,1),
\]
and consequently its quasi-Assouad dimension is equal to $1.$ These results settle an open problem on dimension of graphs posed by Fraser \cite{Fraser2020}.
\end{abstract}
		\maketitle
		\tableofcontents
	
\section{Introduction}

A fundamental objective in fractal geometry is to quantify geometric complexity through appropriate notions of dimension. Although Hausdorff and box-counting dimensions characterise global scaling behaviour, they may fail to capture local irregularities, thereby motivating finer dimensional notions sensitive to local geometry and multiscale structure.

The Assouad dimension is a fundamental tool for quantifying local scaling behaviour. Introduced by Assouad in his 1977 Ph.D. thesis \cite{Assouad}, it captures the extremal local covering properties of a set across locations and scales. The Assouad spectrum, introduced by Fraser and Yu \cite{Fraser}, provides a one-parameter family of dimensions interpolating between the upper box dimension and the quasi-Assouad dimension. For a bounded set $F$, it describes how covering behaviour varies with the prescribed relation between scales, parametrised by $\theta\in(0,1)$. 
For certain fractal sets, the Assouad spectra have been explicitly computed, including overlapping self-affine sets \cite{FraserAffine}, Moran sets \cite{Moran}, stochastically self-similar sets \cite{Sascha}, and Kakaya set \cite{Wang}.

The fractal dimension of the graph of a function is an important measure of its irregularity, but determining it is often difficult. For self-affine FIFs on intervals, Hardin and Massopust \cite{Peter1}  and  Bedford \cite{Bedford} estimated the box dimension of the graph in terms of the vertical scaling factors and related it to the H\"older exponents of the function. B\'ar\'any, Simon, and Solomyak discussed a detailed study of the dimensions of self-similar and self-affine sets \cite{Simon}. While box dimension can often be estimated using scaling factors, Hausdorff dimension is more difficult to determine. 
These challenges motivate the study of generalized affine fractal functions, which provide a flexible framework for investigating finer dimensional properties, including the Assouad spectrum.

The Assouad dimension of graphs of classical fractal functions, particularly the Weierstrass and Takagi functions, remain an active area of research. While the upper box dimension of the Weierstrass graph has been known for several decades \cite{HuLau}, determining its Hausdorff dimension required substantially deeper analysis and the problem was settled by Shen \cite{Weierstrass} in full generality and for certain ranges of parameters by Bara\'{n}ski,  B\'{a}r\'{a}ny and Romanowska \cite{Weiertrass1}. However, we do not know the Assouad or quasi-Assouad dimensions of the graph yet. In particular, Fraser posed the following open problem:\\
\textbf{Question:}[Question 17.11.1]\cite{Fraser2020} What are the Assouad dimensions, quasi-Assouad dimensions, and Assouad spectra of the graphs of the Weierstrass and Takagi functions?\\
These remains an open problem. 
For classes of Takagi functions, partial results and bounds were obtained by Yu. Yu\cite[Theorem 1.1]{Yu2020} showed that the Assouad dimension of $G_T$ can exceed the box dimension. Also, he conjectured that the Assouad dimension is $2$ for all $\lambda\in (0,1)$ and $N\in(1/\lambda,\infty)\cap \mathbb{N})$. Further, Anttila, B\'{a}r\'{a}ny and K\"{a}enm\"{a}ki \cite{Roope}, later gave an exact formula for the general Takagi function of the form 
\[T_\lambda=\lambda^n\sum_{n=0}^{\infty}\text{Dist}(2^nx,\mathbb{Z}),\qquad \frac{1}{2}<\lambda<1,x\in \mathbb{Rb}\]
 showing that its Assouad dimension is strictly less than $2$. Thus, Yu's conjecture does not hold in this case.
 
 However, for the Weierstrass functions, no non-trivial upper or lower bounds for the Assouad dimension of $G_W$ are known.  Howroyd and Yu \cite{Yu1} showed that Assouad dimension of the graph of fractional Brownian motion is almost surely $2$, which tempts to guess that $G_W$ might also have Assouad dimension $2$.  Further support for this idea comes from Chrontsios-Garitsis and Tyson \cite{Tyson2025}, who showed that, after countably many reflections, any Weierstrass graph becomes the graph of a H\'older function with the same exponent and Assouad dimension $2$.  Recently, Feng and Fraser \cite{Feng} also showed that typical H\"older functions have graphs of Assouad dimension $2$. In contrast to these results, Chrontsios-Garitsis recently claims that the Assouad dimension of the general Weierstrass and Takagi function is strictly less than $2$ \cite{Arxive1}. However, there are no explicit lower and upper bounds for the Assouad spectrum of such functions yet computed.

These developments naturally lead to the study of the Assouad spectrum of explicitly constructed fractal functions. In this article, we focus on generalized affine fractal functions, whose graph geometry is governed by affine transformations and scaling parameters. Motivated by previous results on the Assouad dimension of certain fractal sets \cite{Roope, Tyson2025, Yu2020} and using the techniques of controlling the scaling parameter \cite{Ruan1}),  we study the Assouad spectrum and Assouad dimension of the graphs of generalized affine fractal functions. In this direction, our first main result is (see Theorem \ref{Upperbound}):
\begin{theorem}
   	Let  $N\geq 2$ be an integer, and $\Gamma_G$ be the graph of a generalized fractal interpolation function defined as
   	\begin{equation*}
   		f(L_i(x))
   		=
   		\alpha(L_i(x))(f(x)-b(x))+h(L_i(x)), \quad (x\in [x_0,x_N],\, 1\leq i\leq N),
   	\end{equation*}
   	where $x_i-x_{i-1}=\frac{x_N-x_0}{N}$ and  $\alpha: [x_0,x_N]\to (0,1)$  is Lipschitz. Let $D^*$ be the upper box dimension of $\Gamma_G$. Then
   		\begin{enumerate}
   		\item If $N||\alpha||_{\infty}\leq1$, then $ \dim_A^{\theta}(\Gamma_G)=1 \; \text{for}\; \theta\in(0,1).$  Consequently, $ \dim_{qA}(\Gamma_G)=1.$
   		\item  If $N||\alpha||_{\infty}>1$, then   \[\dim_{A}^\theta(\Gamma_G)\leq\min\bigg\{2,D^*+\frac{\theta}{1-\theta}\bigg(1+\frac{\log||\alpha||_{\infty}}{\log N}\bigg)\bigg\}.\]
   			\end{enumerate}
\end{theorem}

Our next main result provides a lower bound for the Assouad spectrum of the graph of generalized affine fractal functions (see Theorem \ref{Lowerbound}):
\begin{theorem}
   	Let $D_*$ be the lower box dimension of $\Gamma_G$. If $N^{-1}<||\alpha||_{\infty}<N^{1-D_*}$, the lower bound of the Assouad spectrum  of $\Gamma_G$ is given by
   		$$\dim_{A}^\theta(\Gamma_G)\geq D_* +\frac{\theta}{1-\theta}\cdot\frac{1-\gamma}{\gamma}\left(\gamma+1-D_*\right) \; \, \text{for}\;\, \theta\in\left(0,\gamma\right),$$
   	where $\gamma=\log_N\left(\frac{1}{||\alpha||_\infty}\right).$ \end{theorem}
   The above results yield explicit numerical bounds that allow us to compute the lower and upper bounds of the Assouad spectrum for a variety of functions, which include the Weierstrass and the Takagi functions, thus addressing the open problem  Question 1. 
The next objective of this paper is to compute a non-trivial lower bound on the Assouad dimension of the graph of general Weierstrass and Takagi functions. 
\begin{theorem}
	Let $N\geq 2$ and $\frac{1}{N}<\lambda<1$ be such that $\lambda^2 N<1$.  Let $$W_{\lambda,N}(x)=\sum_{k=0}^\infty\lambda^k\cos(2\pi N^kx), \;\; x\in \mathbb{R} $$
		denote the Weierstrass function. Then, the Assouad dimension of the graph of $W$ satisfies
		\[
		\dim_A(\Gamma_W)\geq 1 + \log_N\left(\frac{1}{\lambda}\right).
		\]
		\end{theorem}
\begin{theorem}
	Let $N\geq 2$ and $\frac{1}{N}<\lambda<1$ be such that $\lambda^2 N<1$. Let $$T_{\lambda,N}(x)=\sum_{k=0}^\infty\lambda^k\text{Dist}(N^kx,\mathbb{Z}),\;\;  x\in \mathbb{R} $$
	denote the Takagi function. Then, the Assouad dimension  of the graph of $T$ satisfies
	\[
	\dim_A(\Gamma_T)\geq 1 + \log_N\left(\frac{1}{\lambda}\right).
	\]
	\end{theorem}



The remainder of the paper is organized as follows. In Section \ref{sec-Pre}, we review some basic definitions on the generalized affine fractal functions, upper box dimension, Assouad dimension, and Assouad spectrum. In Section \ref{sec-lemma}, we develop several auxiliary results required in the proofs of the main estimates. Section \ref{sec-main} contains the principal results, giving upper and lower bounds for the Assouad spectrum of graphs of generalized affine fractal functions. In Section \ref{sec-ex}, we apply these results to the Weierstrass and Takagi functions and obtain explicit estimates for their Assouad spectra and Assouad type dimensions. We conclude this section by providing the Assouad spectrum of a general self-affine function.

\section{Backgrounds and preliminaries}  \label{sec-Pre}
In this section, we recall some well-known definitions and results that will be used in the next sections. For more details on fractal functions and related works, we refer to \cite{Barnsley1986, Chand, Peter, Maria}. The applications of Assouad spectrum and Assouad dimension can be found in \cite{Fraser2020, FraserAll, Fraser1, Fraser, Lu}.
	Throughout this paper, $N\geq 2$ will denote a positive integer and we shall use  $\Gamma_f=\{(x,f(x)):x\in[x_0,x_N]\}$ as graph of $f:[x_0,x_N]\to \mathbb{R}$ and $ \mathbb{N}_{k}^{0}=\mathbb{N}_{k} \cup \{0\},$ where $\mathbb{N}_k=\{1,2,\dots k\}$.
\subsection{Box dimension}

For any $m_1,m_2\in\mathbb {Z}$ and $\varepsilon>0$, we call
\[
[m_1\varepsilon,(m_1+1)\varepsilon]
\times
[m_2\varepsilon,(m_2+1)\varepsilon]
\]
an \emph{$\varepsilon$-coordinate square} in $\mathbb {R}^2$.

\begin{definition} Let $E\subset \mathbb {R}^2$ be bounded and let $N_E(\varepsilon)$ denote the number of $\varepsilon$-coordinate squares intersecting $E$. The upper box dimension and lower box dimension of $E$ are defined by
\[
\overline{\dim}_B E
=
\limsup_{\varepsilon\to 0^+}
\frac{\log N_E(\varepsilon)}
{\log (1/\varepsilon)}
\]
and
\[
\underline{\dim}_B E
=
\liminf_{\varepsilon\to 0^+}
\frac{\log N_E(\varepsilon)}
{\log (1/\varepsilon)},
\]
respectively.
\end{definition}
If $
\overline{\dim}_B E
=
\underline{\dim}_B E,
$
then we denote the common value by $
\dim_B E$
and call it the box dimension of $E$. It is well known that if $E$ is the graph of a continuous function on a closed interval of $\mathbb {R}$, then $
\dim_B E\geq 1$ (see \cite{Fal}).

\subsection{Assouad dimension and Assouad spectrum}
The Assouad dimension of a metric space measures its thickness and was originally used in embedding theory \cite{Olson}. 
For a bounded set $E\subseteq \mathbb{R}^n$ and $r>0$, let $N_r(E)$ be the minimum number of open balls of radius $r$ required to cover $E$. The \textit{Assouad dimension} of $F \subseteq \mathbb{R}^n$ is given by
\[
\dim_A(F)
=
\inf\left\{
\alpha :
\begin{array}{l}
\text{there exists } C>0 \text{ such that for all } 0<r<R<1 \, \text{and} \, x\in F,\\
N_r\big(B(x,R) \cap F\big) \le C \left( \frac{R}{r} \right)^{\alpha} 
\end{array}
\right\}.
\]

To obtain a more detailed description of scaling between these extremes, Fraser and Yu \cite{Fraser} introduced the Assouad spectrum as: 
\[
\dim_A^\theta(F)
=
\inf\left\{
\alpha :
\begin{array}{l}
\text{there exists } C>0 \text{ such that for all } 0<R<1 \, \text{and} \, x\in F,\\
N_{R^{1/\theta}}\big(B(x,R) \cap F\big) \le C \left( \frac{R}{R^{1/\theta}} \right)^{\alpha}
\end{array}
\right\}.
\]

This spectrum depends on a parameter $\theta \in (0,1)$ and interpolates between the upper box dimension and the quasi-Assouad dimensions. In particular, as $\theta \to 0$, the Assouad spectrum converges to the box-counting dimension. The quasi-Assouad dimension was defined by L\"u and  Xi \cite{Lu}, and  defined as 
  $\dim_{qA}(F)= \lim\limits_{\theta\to 1}\dim_A^\theta(F)$.
  A detailed analysis of the Assouad spectrum and the quasi-Assouad dimension can be found in \cite{Fraser2020, FraserAll}.

\subsection{Generalized affine fractal interpolation functions}
For each $i\in \mathbb{N}_N,$ let  $L_i:[x_0,x_N]\to [x_{i-1},x_i]$
be contractive homeomorphisms satisfying 
	\[L_i(x_0)=x_{i-1}, \; L_i(x_N)=x_i \] and  $
F_i:[x_0,x_N]\times \mathbb{R}\to \mathbb{R}$
be continuous maps satisfying
\begin{equation}\label{eq-2.2}
	F_i(x_0,y_0)=y_{i-1},\qquad 
	F_i(x_N,y_N)=y_i.
\end{equation}
Also, assume that each $F_i$ is uniformly contractive with respect to the second variable, that is, there exists a constant $\beta_i\in(0,1)$ such that 
\begin{equation}\label{eq-2.3}
	\left|F_i(x,y_1)-F_i(x,y_2)\right|
	\leq \beta_i |y_1-y_2|\,  \text{for all}\, x\in[x_0,x_N], \, \text{and all}\, y_1,y_2\in\mathbb{R}.
\end{equation} Further, for $i\in \mathbb{N}_N$, we define the maps $ W_i:[x_0,x_N]\times\mathbb{R}
\to [x_{i-1},x_i]\times\mathbb{R}$
by
\begin{equation}\label{eq-2.4}
	W_i(x,y)=\left(L_i(x),F_i(x,y)\right).
\end{equation}
Here $\{W_i:i\in \mathbb{N}_N\}$ forms an iterated function system (IFS for short) on
$[x_0,x_N]\times\mathbb{R}$. For this IFS, using Barnsley result \cite{Barnsley1986} there exists a unique continuous function $
f:[x_0,x_N]\to\mathbb{R}$ such that its graph $\Gamma_f$
is the invariant set of IFS \eqref{eq-2.4} and satisfies  $
f(x_i)=y_i,\; \text{for}\;  i\in \mathbb{N}_N^0.
$

The function $f$ satisfying these interpolation conditions is called a \textit{fractal interpolation function} (FIF for short) associated with IFS \eqref{eq-2.4}. If $W_i$ is affine, the corresponding FIF is called an affine FIF. Thus, for each $i$, there exist real numbers $a_i,b_i,c_i,d_i,e_i$ such that $
W_i(x,y)=(a_ix+b_i,c_ix+d_iy+e_i).$
Here $d_i$ are called the vertical scaling factors of $f$. For an affine FIF, the function $F_i$ can be rewritten as \[
F_i(x,y)
=d_i(y-b(x))+h(L_i(x)),
\]
where $b$ is a linear function satisfying
	$
	b(x_0)=y_0,\,b(x_N)=y_N,\, h$ is a piecewise linear function such that $h(x_i)=y_i,\,i\in \mathbb{N}^0_N,
	$ and $
	h|_{[x_{i-1},x_i]}
	$
	is linear for each $i\in \mathbb{N}_N.$

 Let $\alpha:[x_0,x_N]\to\mathbb{R}
$
be a continuous function satisfying
$
|\alpha(x)|<1,\,\; \text{for all}\;\;  x\in[x_0,x_N].$ For each $i\in \mathbb{N}_N$, define
\begin{equation*}
	F_i(x,y)
	=
	\alpha(L_i(x))(y-b(x))+h(L_i(x)), \quad i\in \mathbb{N}_N
\end{equation*}
with $b$ and $h$ defined similarly as done above. It is easy to verify that $F_i$ satisfies Equation \eqref{eq-2.2} and Equation  \eqref{eq-2.3}. Hence, if we define
$W_i$ by Equation $\eqref{eq-2.4}$, then the IFS
$
\{W_i\}_{i=1}^{N}
$
determines an FIF $f^*$ satisfying the functional equation
\begin{equation}\label{2.6}
f^*(L_i(x))
=
\alpha(L_i(x))\bigl(f^*(x)-b(x)\bigr)
+h(L_i(x)), \quad x\in [x_0,x_N], i\in \mathbb{N}_N.
\end{equation}
Then $f^*$ is called a \emph{generalized affine FIF}, and $\alpha$ is called
the vertical scaling function of $f^*$.

In this paper, we study the Assouad dimension of a particular class of generalized affine FIFs, where the points $\{x_i\}_{i=0}^{N}$ are uniformly spaced on $[x_0,x_N]$, that is,
\[
x_i-x_{i-1}=\frac{x_N-x_0}{N},\quad i\in \mathbb{N}_N,
\]
and $\alpha$ is positive and Lipschitz. 
That is $\alpha(x)>0$ for  $x\in [x_0,x_N]$ and there exists $\lambda_\alpha\geq 0$ such that 
\[
|\alpha(x_1)-\alpha(x_2)|
\leq \lambda_\alpha |x_1-x_2|
\; \text{for all}\; 
x_1,x_2\in[x_0,x_N].\]

\subsection{Vertical scaling matrices}
In this subsection, we assume that the vertical scaling function \(\alpha\) is Lipschitz with the Lipschitz constant being $\lambda_{\alpha}$. Given a closed interval \(E=[a,b]\), for each $k\in \mathbb{N}$ and
$j\in \mathbb{N}_{N^k}$, we write
\[
E_j^k=
\left[
a+\frac{j-1}{N^k}(b-a),
\,
a+\frac{j}{N^k}(b-a)
\right].
\]

It is clear that
\[
I_j^k=
\left[
\frac{j-1}{N^k},
\frac{j}{N^k}
\right].
\]

For each $i\in \mathbb{N}_N$, we write for simplicity
\[
I_{i,j}^k
=
(I_i)_j^k
=
\left[
\frac{i-1}{N}+\frac{j-1}{N^{k+1}},
\,
\frac{i-1}{N}+\frac{j}{N^{k+1}}
\right].
\]

For $k\in \mathbb{N}$ and $ i\in \mathbb{N}_N, j\in \mathbb{N}_{N^k},$ we define
\[
\overline{\alpha}_{i,j}^k
=
\max_{x\in I_{i,j}^k}
|\alpha(x)|.
\]
It is clear that
$
\overline{\alpha}_{i,j}^k
=
\max_{x\in I_j^k}
|\alpha(L_i(x))|.
$ To calculate the box dimension of the FIF, we introduce an
\(N^k\times N^k\) matrix \(P_k\) as follows. That is, for $ i\in \mathbb{N}_N, \ell \in \mathbb{N}_{N^{k-1}}$ and $ j\in \mathbb{N}_{N^k}$
\begin{equation}\label{3.2}
(\overline{P}_k)_{(i-1)N^{k-1}+\ell,\;j}
=
\begin{cases}
\overline{\alpha}_{i,j}^k,
&
(\ell-1)N < j \leq \ell N,
\\[1ex]
0,
&
\text{otherwise}.
\end{cases}
\end{equation}

Similarly, we define $
\underline{\alpha}_{i,j}^k
=
\min_{x\in I_{i,j}^k}
|\alpha(x)|
$
and define another \(N^k\times N^k\) matrix
\(\underline{P}_k\)
by replacing \(\overline{\alpha}_{i,j}^k\) with \(\underline{\alpha}_{i,j}^k\). Both \(\overline{P}_k\) and \(\underline{P}_k\) are called
\emph{vertical scaling matrices of level \(k\)}.

\medskip
Now we mention a few well-known definitions and results from matrix analysis (see\cite{Horn}). Let $M_n(\mathbb{R})$ denote the set of $n\times n$ real matrices. For $x,y\in \mathbb{R}^n,\; x\geq y $ means  $x_i\geq y_i,\, i\in \mathbb{N}_n $. A matrix $A=(a_{i,j})\in M_n(\mathbb{R})$ is called non-negative if $A\geq 0$ and positive if $A>0$. 
%

A matrix $ A\geq 0$  
is called \emph{irreducible} if there is no permutation matrix $\mathcal{P}$ for which 
\[\mathcal{P}^TA\mathcal{P}=\begin{pmatrix}
A_{11} & A_{12}\\
0 & A_{22}
\end{pmatrix}
 \] with both diagonal blocks nonempty square matrices.

\begin{theorem}[Perron--Frobenius Theorem]
Let \(A\geq 0\) be an \(n\times n\) irreducible  matrix. Then its spectral radius \(\rho(A)\) is an eigenvalue of \(A\), $ \rho(A)>0$ and there exists \(x>0\) such that $Ax=\rho(A)x.$ Also $\rho(A)$ is algebraically simple. 
\end{theorem}
%
%

\begin{theorem}\cite[Theorem 3.4]{Ruan2023}
Assume that the vertical scaling function \(\alpha\) is not identically
zero on every subinterval of $I=[0,1].$
Then $ \rho(\overline{P}_{k+1})
\leq
\rho(\overline{P}_k) \;\text{for all}\; k\in \mathbb{N}.$ Consequently, $
\rho^*:=\lim\limits_{k\to\infty}\rho(\overline{P}_k)
$ exists.
\end{theorem}
\begin{proposition}\cite[Proposition 3.5]{Ruan2023}
Assume that the vertical scaling function \(\alpha\) is positive on \(I\).
Then $
\rho^*=\rho_*
$  and we denote the common value by $\rho_\alpha$.
\end{proposition}
\section{Oscillation sums and box dimension of graphs}\label{sec-lemma}
Let $f:I=[0,1]\to\mathbb {R}$ be continuous. For each  $k\in\mathbb {N}$ and a closed interval $E\subset I$, let
\[
E_j^k
=
E\cap
\left[
\frac{j-1}{N^k},
\frac{j}{N^k}
\right],\quad j\in \mathbb{N}_{N^k}.
\]
For any subset $U\subset I$, define the oscillation of $g$ on $U$ by
\[
\Omega_f(U)
=
\sup_{x_1,x_2\in U}
|f(x_1)-f(x_2)|.
\]
Furthermore, define
\begin{equation}\label{eq:oscillation-sum}
\Omega^k_f(E)
=
\sum_{j=1}^{N^k}
\Omega_f(E_j^k).
\end{equation}
It is clear that the sequence $\{\Omega^k_f(E)\}_{k=1}^{\infty}$ is non-decreasing. 
 The following lemma is vital for establishing a relation between the upper and lower box dimension of the graph of the function by its oscillation.

\begin{lemma}\label{box}\cite[Lemma 4.1]{Ruan2023}
Let $f$ be a continuous function on $I$. Then
\begin{equation*}
\overline{\dim_B}\; \Gamma_f
\leq
1+\limsup_{k\to\infty}\frac{\log(\Omega^k_f(I)+1)}{k\log N},
\end{equation*}
and
\begin{equation*}
\underline{\dim_B}\; \Gamma_f
\geq
1+\liminf_{k\to\infty}\frac{\log(\Omega^k_f(I)+1)}{k\log N}.
\end{equation*}
\end{lemma}
%
%
From this result, we derive a relation between the number of open balls of an arbitrary size $\varepsilon>0$ required to cover $E$ and the oscillation of $f$.
\begin{lemma}
    Let $f$ be a continuous function on $I$. Then for any $\varepsilon>0$ there exist constants $c_1,c_2>0$ such that
    \[
    c_1N^{k}(\Omega^{k}_f(I)+1)
    \leq
    \mathcal{N}_{\varepsilon}(\Gamma_f)
    \leq
    c_2N^{k+1}(\Omega^{k+1}_f(I)+1),
    \]
    where $\frac{1}{N^{k+1}}\leq\varepsilon\leq \frac{1}{N^k}$.
\end{lemma}
\begin{proof}
    By the definition of box-dimension, for $\varepsilon_j=\frac{1}{N^j}$, we have
    \[
    \mathcal{N}_{\varepsilon_k}(\Gamma_f)
    \leq
    \mathcal{N}_{\varepsilon}(\Gamma_f)
    \leq
    \mathcal{N}_{\varepsilon_{k+1}}(\Gamma_f).
    \]
    Applying Lemma \ref{box}, we easily get
    \[
    c_1N^{k}(\Omega^k_f(I)+1)
    \leq
    \mathcal{N}_{\varepsilon}(\Gamma_f)
    \leq
    c_2 N^{k+1}(\Omega^{k+1}_f(I)+1)
    \]
    for some $c_1,c_2>0$.
\end{proof}

\subsection{Oscillation estimate}
The following estimate will be used repeatedly in the subsequent part of the article.
\begin{lemma}\label{Osc}
Let $\alpha$ be a Lipschitz function on $I$. Then there exists a constant $\beta> 0$ such that for 
$i\in \mathbb{N}_N$, every set $Q\subset I_i$, and any $q\in Q$,
\[
\left|
\Omega_f(Q)
-
|\alpha(q)|\,\Omega_f\left(L_i^{-1}(Q)\right)
\right|
\leq
\beta |Q|,\;\, \text{where}\;\, |Q|
=
\sup\{|x_1-x_2|:x_1,x_2\in Q\}
\]
denotes the diameter of $Q$.
\end{lemma}
\begin{proof}
    By the definition of oscillation
    \[
    \Omega_f(Q)=  \sup_{x_1,x_2\in Q} |f(x_1)-f(x_2)|.
    \]
    Using the self-referential relation of $f$, i.e., Equation \eqref{2.6}, we get
    \[
   \Omega_f(Q)=
    \sup_{x_1,x_2\in Q}|h(x_1)-h(x_2) +\alpha(x_1)f(L_i^{-1}(x_1))-\alpha(x_2)f(L_i^{-1}(x_2)).
    \]
     Now, for any $x_1,x_2 \in Q$ we have
    \[
    |h(x_1)-h(x_2)|=|x_1-x_2| |N(y_i-y_{i-1})|\leq N|Q|\max_{j\in \mathbb{N}_N}|y_j-y_{j-1}|
    \]
    and 
    \begin{align*}
         | \alpha(x_1)f(L_i^{-1}(x_1))-\alpha(x_2)f(L_i^{-1}(x_2))|&\leq |\alpha(x_1)-\alpha(q)||f(L_i^{-1}(x_1))|-|\alpha(x_2)-\alpha(q)|\\
         & |f(L_i^{-1}(x_2))| + |\alpha(q)||f(L_i^{-1}(x_1))-f(L_i^{-1}(x_2))|.
    \end{align*}
    Let $B_f:=\max\limits_{x\in I}|f(x)|$ and $\lambda_\alpha$ denote the Lipschitz constant of $\alpha$. Then, we have
  \begin{align*}
         | \alpha(x_1)f(L_i^{-1}(x_1))-\alpha(x_2)f(L_i^{-1}(x_2))|&\leq 2B_f \lambda_{\alpha}|Q|+ |\alpha(q)| \Omega_f\left(L_i^{-1}(Q)\right).
    \end{align*}
    Thus we get
    $\Omega_f(Q)\leq|\alpha(q)|\Omega_f\left(L_i^{-1}(Q)\right)
    +
    \beta |Q|, $
    where  $\beta:=2B_f \lambda_\alpha+ N\max\limits_{j\in \mathbb{N}_N}|(y_j-y_{j-1})|$.
    Using a similar argument, we get
    \[
    \Omega_f(Q)\geq|\alpha(q)|\Omega_f\left(L_i^{-1}(Q)\right)
    -
    \beta |Q|.
    \]

\end{proof}

\begin{lemma}\label{OscLemma}
	Let $f^*$ be the generalized fractal interpolation function with vertical scaling function $\alpha$. If $||\alpha||_\infty\leq\frac{1}{N},$ then $\dim_B\; \Gamma_f=1.$
\end{lemma}

\begin{proof}
	Let $a=||\alpha||_\infty$.
	From Lemma \ref{Osc}, for any $k\in\mathbb{N}$, we have
	\begin{align*}
		\Omega_f(I_j^k)&\leq \alpha \Omega_f(L_j^{-1}(I_j^k)) + \beta \frac{1}{N^k}\\
		&\leq\vdots\\
		&\leq a^k\Omega_f(I) + \beta\left(\frac{a^{k-1}}{N}+ \ldots + \frac{1}{N^k}\right).
	\end{align*}
Since $Na\leq1$, we get
\[
\Omega_f(I_j^k)\leq N^{-k}\Omega_f(I) + \beta(\frac{k}{N^k})= N^{-k}(C_0+k\beta),\; \text{where}\; C_0=\Omega_f(I).
\]
Now, we calculate the upper bound for the oscillation at the $k$-th level, that is,
\begin{align*}
	\Omega_f^k(I)&=\sum_{j=1}^{N^k}\Omega_f(I_j^k)\\
	&\leq N^k\cdot N^{-k}(C_0+k\beta)\\
	&=C_0+k\beta.
\end{align*}
Therefore, by Lemma \ref{box}, we have $\dim_B(\Gamma_f)\leq 1$. Also, for any continuous function $f$ on $I,$ we get $ \dim_B(\Gamma_f)\geq 1.$   
 \end{proof}

\section{Main results}\label{sec-main}
In this section, we prove the main results of this paper. First, we compute  the upper bound of the Assouad spectrum of the graph of the generalized affine fractal functions. We use the covering techniques and oscillation estimates as proved in Lemma \ref{Osc} and Lemma \ref{OscLemma}. Then we provide an estimation of the lower bound of the associated graph under some restrictions on the scaling functions.  
   \begin{theorem}\label{Upperbound}
	Let $\Gamma_G$ be the graph of a generalized fractal interpolation function defined in Equation \eqref{2.6}, 
   	where $x_i-x_{i-1}=\frac{x_N-x_0}{N}$ and  $\alpha: [x_0,x_N]\to (0,1)$   is Lipschitz. Let $D^*$ be the upper box dimension of $\Gamma_G$. Then   
   		\begin{enumerate}
   		\item If $N||\alpha||_{\infty}\leq1,$ then $ \dim_A^{\theta}(\Gamma_G)=1 \; \text{for}\; \theta\in(0,1). $ Consequently, $ \dim_{qA}(\Gamma_G)=1.$
   		\item  If $N||\alpha||_{\infty}>1$, then  \[\dim_{A}^\theta(\Gamma_G)\leq\min\bigg\{2,D^*+\frac{\theta}{1-\theta}\bigg(1+\frac{\log||\alpha||_{\infty}}{\log N}\bigg)\bigg\}.\]
   			\end{enumerate}
   	 \end{theorem}
    \begin{proof}	
   	Let $\theta_0=\min\left(\frac{2-D^*}{2-D^*+ \log_N(N||\alpha||_\infty)},1\right)$. Consider any $R>0$ and $\theta\in(0,\theta_0)$. Let $k,q\in \mathbb{N}$ such that
   	$$\frac{1}{N^{k+1}}\leq R\leq \frac{1}{N^k},\; \frac{1}{N^{q+1}}\leq R^{1/\theta}\leq \frac{1}{N^q}.$$
   	
   	By the covering estimate obtained from the oscillation method, there exists $C>0$ such that
   	\[
   	\mathcal{N}_{R^{1/\theta}}((\Gamma_G))
   	\leq
   	C N^q\bigl(\Omega^{q}_f(I)+1\bigr)
   	=CN^q\left(\sum_{|v|=k}\Omega^{q-k}_f(I_v)+1\right).
   	\]
   	Let \(w\in\mathbb{N}^k\) be a word of length \(k\). Since $F_w(\Gamma_G)$ is the level $k$ 
   	portion of the graph, from the oscillation estimate we have
   	\begin{align*}
   		\mathcal{N}_{R^{1/\theta}}(F_w(\Gamma_G))&
   		\leq
   		CN^{q-k}\bigl(N^k\Omega^{q-k}_f(I_w)+1\bigr)\\
   		&\leq C N^{q}\bigl(\Omega^{q-k}_f(I_w)+N^{-k}\bigr).
   	\end{align*}
   	\begin{figure}[H]
   		\centering
   		
   		\begin{minipage}{0.4\textwidth}
   			\centering
   			\includegraphics[width=\textwidth]{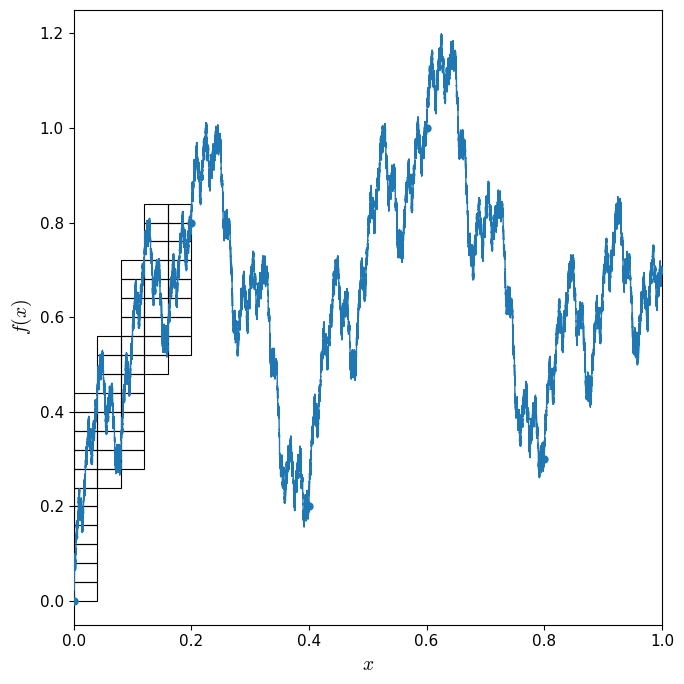}
   			\caption{Affine fractal interpolation function with
   				$\alpha_i=0.4$, $1\leq i\leq 5$.}
   			\label{fig:full}
   		\end{minipage}
   		\hfill
   		\begin{minipage}{0.4\textwidth}
   			\centering
   			\includegraphics[width=\textwidth]{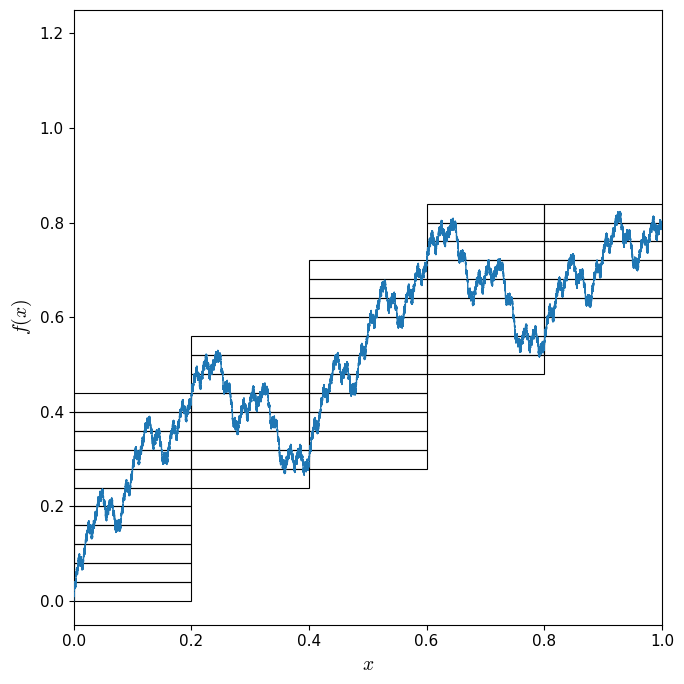}
   			\caption{Rescaled first cylinder with the corresponding
   				box covering.}
   			\label{fig:part}
   		\end{minipage}
   		
   	\end{figure}
   	By the definition of oscillation,
   	\[
   	\Omega^{q-k}_f(I_w)
   	=
   	\sum_{|v|=q-k} \Omega_f(I_{wv}).
   	\]
   	
   	Let $a:= ||\alpha||_{\infty}$. Applying the oscillation estimate obtained in Lemma \ref{Osc} recursively \(k\) times to each interval \(I_{wv}\), we obtain
   	\begin{align*}
   		\Omega_f(I_{wv})
   		&\leq a\Omega(f,I_{w'v})+ \beta N^{-q}\\
   		&\vdots\\
   		&\leq a^k \Omega(f,I_v)
   		+
   		\beta\sum_{r=0}^{k-1}a^rN^{-(q-r)}.
   	\end{align*}

   	Summing over all words \(v\) of length \(q-k\) gives
   	\begin{align*}
   		\Omega^{q-k}_f(I_w)
   		&=
   		\sum_{|v|=q-k}\Omega_f(I_{wv})\\
   		&\leq
   		a^k\sum_{|v|=q-k}\Omega_f(I_v)
   		+
   		\beta N^{q-k}\sum_{r=0}^{k-1}a^rN^{-(q-r)}\\
   		&=
   		a^k \Omega^{q-k}_f(I)
   		+
   		\beta N^{-k}\sum_{r=0}^{k-1}(aN)^r.
   	\end{align*}
   	
   	Hence,
   	\begin{equation}\label{eq:oscillation-cylinder}
   		\Omega^{q-k}_f(I_w)
   		\leq
   		a^k \Omega^{q-k}_f(I)
   		+
   		\beta N^{-k}\sum_{r=0}^{k-1}(aN)^r.
   	\end{equation}
   	Using the oscillation-box counting estimate, we get
   	\[
   	\Omega^{q-k}_f(I)
   	\leq
   	2N^{k-q}
   	\mathcal{N}_{N^{k+1}R^{1/\theta}}(G),
   	\]
   	and the fact that for any $\varepsilon>0$ there exists $C_{b}>0$ such that
   	\[
   	\mathcal{N}_{\varepsilon}(G)
   	\leq
   	C_b\varepsilon^{-D^*},
   	\]
   	which gives
   	\[
   	\Omega^{q-k}_f(I)
   	\leq
   	C_b'\,N^{k-q}
   	\bigl(N^{k+1}R^{1/\theta}\bigr)^{-D^{*}},
   	\]
   	where $C_b'=2C_b$. 
   	We now distinguish three cases:\\
   	\medskip
   	\noindent\textbf{Case 1:} Suppose $a N<1.$\\
   	Since $
   	\sum_{r=0}^{k-1}(aN)^r
   	\leq
   	\frac{1}{1-aN},
   	$
   	it follows from Equation \eqref{eq:oscillation-cylinder} that
   	\[
   	\Omega^{q-k}_f(I_w)
   	\leq
   	a^k \Omega^{q-k}_f(I)
   	+
   	\beta_1 N^{-k}
   	\leq
   	N^{-k} (\Omega^{q-k}_f(I)
   	+
   	\beta_1),
   	\]
   	where $\beta_1=\frac{\beta}{1-aN}$. Hence, we get
   	\begin{align*}
   		\mathcal{N}_{R^{1/\theta}}(F_w(G))
   		&\leq
   		N^q
   		\Bigl(
   		N^{-k}\Omega^{q-k}_f(I)
   		+
   		\beta_1 N^{-k}
   		+
   		N^{-k}
   		\Bigr)\\
   		&\leq N^q
   		\Bigl(
   		N^{-k}\Omega^{q-k}_f(I)
   		+
   		\beta' N^{-k}
   		\Bigr),
   	\end{align*}
   	where $\beta'=\beta_1 +1$. Putting this in the oscillation estimate, we get
   	\begin{align*}
   		\mathcal{N}_{R^{1/\theta}}(F_w(G))
   		&\leq
   		C_b'\,N^{q-k}
   		N^{k-q}
   		\left(N^{k+1}R^{1/\theta}\right)^{-D^*}
   		+
   		\beta 'N^{q-k}
   		\\
   		&\leq C_b'\,
   		\left(\frac{R}{R^{1/\theta}}\right)^{D^*}
   		+
   		\beta'	\left(\frac{R}{R^{1/\theta}}\right).
   	\end{align*}
   	Note that if $aN<1$, then by Lemma \ref{OscLemma}, we have $D^*=1.$ Thus, we get
   	\begin{align*}
   		\mathcal{N}_{R^{1/\theta}}(F_w(G))
   		\leq
   		C_\beta	\left(\frac{R}{R^{1/\theta}}\right),
   	\end{align*}
   	where $C_\beta=C_b'+\beta$.
   	Therefore, in this case, for all $\theta\in(0,1)$, we have
   	$
   	\dim_A^{\theta}(G)\leq 1
   	$ as $\theta_0=1$. Since $\dim^\theta_A(G)\geq \overline\dim_B(G)=1$, we have $
   	\dim_A^{\theta}(G)=1 .$ Hence
   	$\dim_{qA}(G)=1$.\\
   	
   	\noindent\textbf{Case 2:} Suppose \(aN=1\). \\
   	In this case,
   	\[
   	\sum_{r=0}^{k-1}(aN)^r=k.
   	\]
   	Thus
   	\[
   	\Omega^{q-k}_f(I_w)
   	\leq
   	a^k \Omega^{q-k}_f(I)
   	+
   	\beta_1 k N^{-k}
   	\leq
   	N^{-k} (\Omega^{q-k}_f(I)
   	+
   	\beta_1k).
   	\]
   	Also, note that $aN=1$ implies $\rho^*\leq1$ which again means $D^*=1$. Hence, by a similar calculation we get
   	\begin{align*}
   		\mathcal{N}_{R^{1/\theta}}(F_w(G))
   		&\leq C_b'\,
   		\left(\frac{R}{R^{1/\theta}}\right)
   		+
   		\beta'k	\left(\frac{R}{R^{1/\theta}}\right) .
   	\end{align*}
   	Note that $k\leq|\log_N R|$. For any  $\varepsilon>0$ there exists $R_0>R>0$  such that
   	\[ |\log_NR|\leq R^{-\varepsilon},\, \text{for} \; 0<R<R_0.\] Putting $R^{-\varepsilon}= \left(\frac{R}{R^{1/\theta}}\right)^{\frac{\varepsilon} {1/\theta-1}}$ we get
   	\begin{align*}
   		\mathcal{N}_{R^{1/\theta}}(F_w(G))
   		&\leq C_b'\,
   		\left(\frac{R}{R^{1/\theta}}\right)
   		+
   		\beta'	\left(\frac{R}{R^{1/\theta}}\right)^{1+\frac{\varepsilon} {1/\theta-1}} \\
   		&\leq (C_b'+\beta')	\left(\frac{R}{R^{1/\theta}}\right)^{1+\frac{\varepsilon} {1/\theta-1}}.
   	\end{align*}
   	Since $\varepsilon$ is arbitrary, we conclude that $\dim_A^{\theta}(G)\leq1$ for all $\theta\in(0,1) $. Thus  $\dim_A^{\theta}(G)=1$ which in term gives $\dim_{qA}(G)=1$.\\
   	\noindent
   	\textbf{Case 3:} Suppose $aN>1.$\\
   	Since
   	\[
   	\sum_{r=0}^{k-1}(aN)^r
   	=
   	\frac{(aN)^k-1}{aN-1}
   	\leq
   	\frac{(aN)^k}{aN-1},
   	\]
   	by Equation \eqref{eq:oscillation-cylinder}, we get
   	\begin{equation*}
   		\Omega^{q-k}_f(I_w)
   		\leq
   		a^k \Omega^{q-k}_f(I)
   		+
   		\beta' N^{-k}(aN)^k=
   		a^k \Omega^{q-k}_f(I)
   		+
   		\beta_2 a^k,
   	\end{equation*}
   	
   	where $\beta_2= \frac{\beta}{aN-1}$. Then
   	\[
   	\mathcal{N}_{R^{1/\theta}}(F_w(G))
   	\leq
   	N^q
   	\Bigl(
   	a^{k}\Omega^{q-k}_f(I)
   	+
   	\beta_2 a^k
   	+
   	N^{-k}
   	\Bigr).
   	\]
   	Substituting this estimate yields
   	\begin{align}\label{eq-inequalityN}
   		\mathcal{N}_{R^{1/\theta}}(F_w(G))
   		&\leq
   		C_b'\,a^{k}
   		N^{k-D^*(k+1)}
   		R^{-D^*/\theta}
   		+
   		\beta_2N^qa^k
   		+
   		N^{q-k}\\ \nonumber
   		&\leq C_b' N^{k-D^*}  a^k \bigg(\frac{R}{R^{1/\theta}}\bigg)^{D^*}
   		+ \beta_2\bigg(\frac{R}{R^{1/\theta}}\bigg) (Na)^k
   		+
   		N^{q-k}.
   	\end{align}	
   	Assume $\bigg(\frac{R}{R^{1/\theta}}\bigg)^{A'}=N^{k-D^*}a^k$ for some $A'\in\mathbb{R}$. Then, we can write
   	\begin{align*}
   		(k-D^*)\log N + k \log a&= A' \log\left(\frac{R}{R^{1/\theta}}\right)\\
   		&=A'(1-1/\theta)\cdot\log R \\
   		\implies A'=\frac{(k-D^*)\log N + k \log a}{(\frac{1}{\theta}-1)\log \frac{1}{R}}.
   	\end{align*}
   	Since $R\leq N^{-k}$, we have
   	
   	\begin{align*}
   		A'&\leq\frac{(k-D^*)\log N + k \log a}{(\frac{1}{\theta}-1) k\log N}\\
   		&= \cfrac{\theta}{1-\theta}\left(1-\cfrac{D^*}{k}+\cfrac{ \log a}{ \log N}\right):=A_k-D^*.
   	\end{align*}
   	Since $A_k$ is an increasing sequence, let $A= \lim\limits_{k\to \infty} A_k$. Then
   	\[
   	A =
   	D ^*+ \frac{\theta}{1-\theta}
   	\left(1 + \frac{\log a}{\log N}\right).
   	\]
   	
   	\noindent Next, assume $(Na)^k=\bigg(\frac{R}{R^{1/\theta}}\bigg)^{B}.$ By a similar argument, we get
   	\[
   	B=1 + \frac{\theta}{1-\theta}
   	\left(1 + \frac{\log a}{\log N}\right).
   	\]
   	Note that $A\geq B\geq1$, since $D^*\geq1$. Putting these expressions in  \eqref{eq-inequalityN}, we obtain
   	\begin{align*}
   		\mathcal{N}_{R^{1/\theta}}(B(x,R)\cap G)
   		&\leq
   		C_b' \left(\frac{R}{R^{1/\theta}}\right)^{A}
   		+
   		\beta_2 \left(\frac{R}{R^{1/\theta}}\right)^{B}
   		+
   		\left(\frac{R}{R^{1/\theta}}\right) \\  
   		&\leq C_\beta'' \left(\frac{R}{R^{1/\theta}}\right)^{A},
   	\end{align*}
   	where $C_\beta''= C_b'+\beta_2+1$. Thus, we have \[
   	\dim_A^{\theta}(G)\leq A  \, \text{for}\; \theta\in(0,\theta_0),\; \text{  where}\; \theta_0= \dfrac{2-D^*}{2-D^*+ \log_N(N||\alpha||_\infty)}.\]
   	Note that $A(\theta_0)=2$, and so we get a trivial upper bound of $2$ for $\theta\in[\theta_0,1)$. Therefore, the Assouad spectrum is given by $
   	\dim_A^\theta G
   	\le
   	\min\left\{
   	2,
   	A
   	\right\}.
   	$
   \end{proof}
   The next result provides an estimation of the lower bound of the Assouad spectrum of $\Gamma_G$.
   \begin{theorem}\label{Lowerbound}
   	Let $\Gamma_G$ be the graph of the  generalized fractal interpolation function defined in Equation \eqref{2.6}.
   	Let $D_*$ be the lower box dimension. If $N^{-1}<||\alpha||_{\infty}<N^{1-D_*}$, the lower bound on the Assouad spectrum is given by
   	$$\dim_{A}^\theta(G)\geq D_* +\frac{\theta}{1-\theta}\frac{1-\gamma}{\gamma}\left(\gamma+1-D_*\right) \; \, \text{for}\;\, \theta\in\left(0,\gamma\right),$$
   	where $\gamma=\log_N\left(\frac{1}{||\alpha||_\infty}\right).$
   \end{theorem}
   \begin{proof}
    Let $a=||\alpha||_{\infty}$. Consider any $R>0$ and let $k,q\in \mathbb{N}$ such that $a^{k+1}\leq R\leq a^{k},$
   	and
   	\[\frac{1}{N^{q+1}}\leq R^{1/\theta}\leq \frac{1}{N^q}.\]
   	Also consider $k'\in \mathbb{N}$ such that \[ N^{-k-1}\leq R\leq N^{-k'}.\]
    Let $x_{\max}\in I$ be the point such that $\alpha(x_{\max})=a$. Let $\omega$ be the word of length $k$ such that $x_{\max}\in L_\omega(I)$. The number of words $w$ such that $\|w|=k$ and   $F_{w}\subset B(x_{\max},R)\cap \Gamma_G$ is atleast $Ka^{k'-k}$ for some constant $K>0$. Consider any such word $w$, then
   	\[
   	\mathcal{N}_{R^{1/\theta}}(F_w(\Gamma_G))
   	\geq
   	CN^q\bigl(\Omega^{q}_f(I_w)\bigr),
   	\] 
   	for some $C>0$. By definition,
   	\[
   	\Omega^{q-k}_f(I_w)
   	=
   	\sum_{|v|=q-k} \Omega_f(I_{wv}).
   	\]
   	For the argument to be valid, we require $q>k$. Simplifying this, we get
   	\[
   	q>k\implies \frac{1}{\theta}\frac{\log 1/R}{\log N}> \frac{\log 1/R}{\log 1/a}\implies \theta<\frac{\log\frac{1}{a}}{\log N}.
   	\]
   	Hence, the permissible range for $\theta$ lies in $\left(0,\frac{\log\frac{1}{a}}{\log N}\right)$. Choose an arbitrary $t\in L_w'(I)$ where $w'$ is the word of length $k'$ and $x_{max}\in L_w'(I) $. Since $\alpha$ is Lipschitz, we can write
   	$$\alpha(t)\geq a-\frac{\lambda_{\alpha}}{N^{k'}}=:a'$$
      	Applying the oscillation estimate recursively \(k\) times to each interval \(I_{wv}\) for any $t\in L_wv(I)$, we obtain
   	\begin{align*}
   		\Omega_f(I_{wv})
   		&\geq a'\Omega(f,I_{w'v})- \beta N^{-q}\\
   		&\vdots\\
   		&\geq a'^k \Omega(f,I_v)
   		-
   		\beta\sum_{r=0}^{k-1}a^rN^{-(q-r)}.
   	\end{align*}
   	Summing over all words \(v\) of length \(q-k\), we have
   	\begin{align*}
   		\Omega^{q-k}_f(I_w)
   		&=
   		\sum_{|v|=q-k}\Omega_(I_{wv})\\
   		&\geq
   		a'^k\sum_{|v|=q-k}\Omega_f(I_v)
   		-
   		\beta N^{q-k}\sum_{r=0}^{k-1}a'^rN^{-(q-r)}\\
   		&=
   		a'^k \Omega^{q-k}_f(I)
   		-
   		\beta N^{-k}\sum_{r=0}^{k-1}(a'N)^r.
   	\end{align*}
   Evaluating the error term for small enough $R$, we get
   	\[
   	\sum_{r=0}^{k-1}(aN)^r
   	=
   	\frac{(aN)^k-1}{aN-1},
   	\]
   	 and hence the inequality becomes
   	\[
   	\Omega^{q-k}_f(I_w)\geq a'^k \Omega^{q-k}_f(I)
   	-
   	\beta_1 (a'^k-N^{-k}),
   	\]
   	where $\beta_1=\frac{\beta}{a'N-1}.$
   	Replacing the oscillation with box-counting, we get
   	\[
   	\Omega^{q-k}_f(I)
   	\geq
   	N^{k-q}
   	\mathcal{N}_{N^{k}R^{1/\theta}}(G)-1,
   	\]
   	along with
   	$
   	\mathcal{N}_{\varepsilon}(\Gamma_G)
   	\geq
   	C_b\varepsilon^{-D_*},
   	$
   	gives
   	\[
   	\Omega^{q-k}_f(I)
   	\geq
   	C_b\,N^{k-q}
   	\bigl(N^{k}R^{1/\theta}\bigr)^{-D_*}-1.
   	\]
   	\begin{align*}
   		\mathcal{N}_{R^{1/\theta}}(B(x_{\max},R)\cap \Gamma_G)
   		&\geq
   		C_K a^{k'-k}\left(a-\frac{\lambda_\alpha}{N^{-k'}}\right)^{k} 
   		N^{k-kD_*}
   		R^{-D_*/\theta}-\beta_1N^q(a'^k)+\beta_1 N^{q-k}\\
   		&\geq
   			C_K a^{k'}\,N^{k-kD_*}
   		R^{-D_*/\theta}- \lambda_\alpha C_K a^{k'-k}ka^{k-1}N^{-k'}N^{k-kD_*}
   		R^{-D/\theta}\\
   		 &-\beta_1N^q(a^k)+\beta_1 N^{q-k}.
   	\end{align*}
   	Define $ \gamma=\frac{\log\frac{1}{a}}{\log N}$.
   	Since $a^{k'}\asymp R^{\gamma}$, we can write the main term as follows
   	\[
    a^{k'}\,N^{k-kD_*}R^{-D_*/\theta}= K_1 R^\gamma R^{\frac{1}{\gamma}(D_*-1)} R^{-D_*/\theta}= K_1\left(\frac{R}{R^{1/\theta}}\right)^{A_L},
   	\]
   	where $A_L = D_* + \frac{\theta}{1-\theta}\frac{1-\gamma}{\gamma}\left(\gamma+1-D_*\right)$. Next, we consider the second term of the inequality
   	 	\[
   	a^{k'-k}ka^{k-1}N^{-k'}N^{k-kD_*}
   	R^{-D_*/\theta}= K_2 R^{-\varepsilon}  R^\gamma  R R^{\frac{1}{\gamma}(D_*-1)} R^{-D_*/\theta}=K_1\left(\frac{R}{R^{1/\theta}}\right)^{B_L},
   	\]
   	where $B_L = D_* + \frac{\theta}{1-\theta}\left[\frac{1-\gamma}{\gamma}\left(\gamma+1-D_*\right)-1+\varepsilon\right].$ Since $\varepsilon\to 0$ as $k\to \infty,$ we will have $B_L<A_L.$ Moving on to the third term, we have
   \[
   N^qa'^{k}\leq N^qa^k=K_3\left(\frac{R}{R^{1/\theta}}\right).
   \]
   	Since $aN>N^D_*$, we have $A_L>1$. The fourth term is dominated by the third term because  $N^qa^k\geq N^{q-k}$. Therefore we have
   	\[
   	\mathcal{N}_{R^{1/\theta}}(B(x_{\max},R) \cap \Gamma_G)\geq C_3\bigg(\frac{R}{R^{1/\theta}}\bigg)^{A_L}
   	\]
   	for $C_3=K_1C_b>0$. Thus we have $
   	\dim_A^{\theta}(\Gamma_G)\geq A_L.
   	$
   	
   \end{proof}
Thus we have established non-trivial numerical bounds on the Assouad spectrum of such functions which are related to the parameters of the scaling function. We combine this results into a corollary when the upper and  ower box dimension coincide:

\begin{corollary}\label{cor}
	Let $\Gamma_G$ be the graph of the  generalized fractal interpolation function defined in Equation \eqref{2.6}.
   	Let $D_*, D^*$ be the lower and upper box dimensions of $\Gamma_G$ and assume that $D^*=D_*$ and denote the box dimension by $D$. If $||\alpha||_\infty<N^{1-D},$ then for $\theta\in(0,\gamma)$ we have
	\[
	D+ \frac{\theta}{1-\theta}\cdot\frac{1-\gamma}{\gamma}\left(\gamma+1-D\right)\leq \dim^\theta_{A}(\Gamma_G)\leq D +\frac{\theta}{1-\theta}\left(1-\gamma\right),
	\]
	where $ \gamma=\dfrac{\log\frac{1}{a}}{\log N}$.
\end{corollary}

\begin{remark}
	We want to emphasize the fact that the proof for the lower bound works even when we consider the case $N||\alpha||_\infty\geq N^D$.  That is we have:
	\[
	N||\alpha||_\infty\geq N^D\implies \gamma<D-1,
	\]
	and hence:
	\[
	\dim^\theta_{A}(\Gamma_G)\geq D+ \frac{\theta}{1-\theta}\cdot\frac{1-\gamma}{\gamma}\left(\gamma+1-D\right)\leq D.
	\]
	But we already know that 
	\[
	\dim^\theta_{A}(\Gamma_G)\geq D,
	\]
	and hence we do not obtain new information when $N||\alpha||_\infty\geq N^D$.
\end{remark}

\section{Applications}\label{sec-ex}
In this section, we consider two classic examples of fractal functions, namely, Weierstrass and Takagi functions. The Hausdorff dimension and box-counting dimension of these functions have been extensively studied in \cite{Weiertrass1, Ruan1, Weierstrass, Takeo1996}.

In this article, for the first time, we produce general results regarding the Assouad dimension and Assouad spectrum of these fractal functions.
\subsection{Weierstrass functions}
Let $N\geq 2$ and $\frac{1}{N}<\lambda<1$, and let $$W_{\lambda,N}^\phi(x)=\sum_{k=0}^\infty\lambda^k\phi(N^kx),\quad x\in \mathbb{R}$$ denote the 
 Weierstrass function, where $\phi:\mathbb{R}\to \mathbb{R}$ is a $\mathbb{Z}$-periodic real analytic function. It is shown that such a function is either real analytic, or the Hausdorff dimension of its graph $\Gamma_W$ is $2+\log_N\lambda$ (\cite{RenShen1}). Let $\phi(x)=\cos(2\pi x)$. This gives us the classical Weierstrass function with equal Hausdorff and box-dimension as $2+\log_N\lambda$.
Consider $f=W_{\lambda,N}^\phi|_{[0,1]}$. We can represent $f$ as a generalized affine fractal interpolation function since for $1\leq n\leq N$
\[
f\bigg(\frac{x+n-1}{N}\bigg)=\cos\bigg(\frac{2\pi (x+n-1)}{N}\bigg)+ \lambda\sum_{k=0}^\infty\lambda^k\cos(2\pi N^kx)=\cos\bigg(\frac{2\pi (x+n-1)}{N}\bigg)+ \lambda f(x),
\]
where $\lambda$ becomes the vertical scaling function and hence $||\alpha||_{\infty}=\lambda$. In view of Theorem \ref{Upperbound} and Theorem \ref{Lowerbound}, we prove: 
\begin{theorem}\label{Weier} 
	Let $$W_{\lambda,N}(x)=\sum_{k=0}^\infty\lambda^k\cos(2\pi N^kx),\quad x\in \mathbb{R}$$
	denote the Weierstrass function. Consider $f=W_{\lambda,N}|_{[0,1]}$ and denote the graph of $f$ by $\Gamma_W$. Let  $\gamma=\log_N\frac{1}{\lambda}$. Then:
\begin{enumerate}
		\item \[
		\dim_A^\theta(\Gamma_W)\leq
		\frac{2-\gamma-\theta}{1-\theta}, \quad \theta\in(0,\gamma).\]
		\item  If $\lambda^2N<1$, then we obtain a lower bound on the spectrum given by
		\[
		\dim_A^\theta(\Gamma_W)\geq 2 -\gamma+ \frac{\theta}{1-\theta}\frac{1-\gamma }{\gamma}\left(2\gamma-1\right), \quad \theta\in(0,\gamma).\]
		\item If $\lambda^2N<1$, then the Assouad dimension is bounded below by
		\[
		\dim_A(\Gamma_W)\geq 1 + \gamma.
		\]
\end{enumerate}
\end{theorem}

\begin{proof}
	(1) follows from Theorem \ref{Upperbound} and (2) follows from Theorem \ref{Lowerbound}. To see (3), note that 
		\begin{align*}
		\dim_A(\Gamma_W)
		&\geq\lim_{\theta\to \log_N(\frac{1}{\lambda})}\dim_A^\theta(\Gamma_W)\\
		&
		\geq2 -\log_N 1/\lambda+ 2\log_N 1/\lambda-1= 1+\log_N 1/\lambda.
		\end{align*}
\end{proof} In the following example, we derive exact estimates for the Assouad spectrum and Assouad dimension of the classical Weierstrass function.
\begin{example}\label{ex-5.1}
	Let $W_{\lambda,N}^\phi$ denote the classical Weierstrass function where $\phi(x)=\cos(2\pi x)$. Let $\lambda=0.6$ and $N=3.$ Then $\log_N\frac{1}{\lambda}=0.630>1/2\ldots$,
    and hence we have
	\[
\dim_A^\theta(\Gamma_W)\leq\begin{cases}
	\frac{1.369-\theta}{1-\theta}, \qquad &\theta\in(0,0.630\ldots) \\
	2, &\theta\in[0.630\ldots,1)
\end{cases}
\]
and for $\theta\in(0,0.630\ldots)$
\[
\dim_A^\theta(\Gamma_W)\geq 1.369 + \frac{\theta}{1-\theta}\left(\frac{0.369}{0.630}\right){0.261}.
\]
As a result, $\dim_A(\Gamma_W)\geq 1.630\ldots$.

\begin{figure}[H]
	\centering
	
	\begin{minipage}{0.48\textwidth}
		\centering
	\begin{tikzpicture}
		\begin{axis}[
			width=\textwidth,
			height=6cm,
			xlabel={$x$},
			ylabel={$W_{\lambda,N}^{\phi}(x)$},
			grid=major,
			enlarge x limits=false,
			enlarge y limits=false,
			line join=round,
			]
			
			\addplot[
			thick,
			blue
			]
			table[x index=0,y index=1] {weierstrass1.dat};
			
		\end{axis}
	\end{tikzpicture}
	\caption{
		Classical Weierstrass function with
		$\lambda=\frac12$,
		$N=3$.
	}
	\end{minipage}
	\hfill
	\begin{minipage}{0.48\textwidth}
		\centering
		\begin{tikzpicture}
		\begin{axis}[
			width=\textwidth,
			height=6cm,
			axis lines=middle,
			xlabel={$\theta$},
			ylabel={$\dim_A^\theta(G)$},
			xmin=0, xmax=1,
			ymin=1, ymax=2.2,
			samples=200,
			clip=true,
			]
			
			\addplot[darkgray, thick, domain=0.0001:0.630] 
			{(1.369-x)/(1 - x)};
			\addplot[lightgray, thick, domain=0.0001:0.630] 
			{(1.369-x*1.217)/(1 - x)};
			\addplot[darkgray, thick, domain=0.630:1] 
			{2};
			\addplot[dashed, thin ,domain=1:2] ({0.630}, x);
		\end{axis}
	\end{tikzpicture}
	\caption{Upper and lower bounds of the Assouad spectrum of the Weierstrass function}
	\end{minipage}
	
\end{figure}

\end{example}
\begin{remark}
	In a recent work, Chrontsios and Tyson \cite{Tyson2025} deals with the Assouad spectrum of H\"older graphs and mentions the upper bound of the regularized Assouad spectrum for a function $f:I\to\mathbb{R}$ (where $I\subset\mathbb{R}$ is a closed interval) as:
	\[
	\dim^{\theta}_{A,reg}(\Gamma_f)
	\leq
	 \frac{2-\alpha-\theta}{1-\theta},\quad \theta\in(0,\alpha)
	\] 
	
 We also note that the Weierstrass function is an  $\alpha$-H\"older function where $\alpha=-\log_N\lambda$. Our findings are consistent with those of this work, as our methodology yields the same upper bound. In addition, we obtain a non-trivial lower bound.
\end{remark}
\subsection{Takagi functions}
Let $\phi(x)=\text{Dist}(x,\mathbb{Z)}$, where
\[
\text{Dist}(x,\mathbb{Z)}=\begin{cases}
	\{x\},& 0\leq\{x\}<\frac{1}{2}\\
	1-\{x\},&\frac{1}{2}\leq\{x\}\leq1.
\end{cases}
\]
Here $\{x\}$ denotes the fractional part of $x$. Then $T_{\lambda,N}^\phi$ gives us the generalized Takagi function. It is known that the box dimension of such functions is $2+\log_N\lambda$ when $\lambda N>1$ and $1$ when $\lambda N\leq 1$ (see \cite{Weiertrass1}).
\par
Let $g=T_{\lambda,N}^\phi|_{[0,1]}$ and denote the graph of $g$ by $\Gamma_T$. Similar to the Weierstrass case, we can write the functional equation of $g$ as \[
g\bigg(\frac{x+n-1}{N}\bigg)=\text{Dist}\bigg(\frac{ x+n-1}{N},\mathbb{Z}\bigg)+ \lambda\sum_{k=0}^\infty\lambda^k\text{Dist}(N^kx,\mathbb{Z})=\text{Dist}\bigg(\frac{ x+n-1}{N},\mathbb{Z}\bigg)+ \lambda g(x).
\]
As before, we get
\[
\dim_A^\theta(\Gamma_T)\leq\begin{cases}
		h(\theta), & \lambda N>1\\
		1, & \lambda N\leq1,
\end{cases}
\]
where
\[
h(\theta)=\begin{cases}
	\frac{2+\log_N \lambda - \theta}{1-\theta}, & \theta\in(0,\log_N(\frac{1}{\lambda}))\\
	2, & \theta\in[\log_N(\frac{1}{\lambda}),1).
\end{cases}
\]
Also, when $\lambda^2N<1$
\[
\dim_A^\theta(\Gamma_T)\geq 2 -\log_N 1/\lambda+ \frac{\theta}{1-\theta}\frac{1-\log_N 1/\lambda }{\log_N 1/\lambda}\left(2\log_N 1/\lambda-1\right).
\] 
So we a have similar result for Takagi function and the proof is almost identical to the proof of Theorem \ref{Weier}.
\begin{theorem}\label{Takag}
	Let \[T_{\lambda,N}(x)=\sum_{k=0}^\infty\lambda^k\textrm{Dist}(N^kx,\mathbb{Z}),\; x\in\mathbb{R}\]
	denote the general Takagi function. Consider $g=T_{\lambda,N}|_{[0,1]}$ and denote the graph of $g$ by $\Gamma_T$. Let  $\gamma=\log_N\frac{1}{\lambda}$, then:
	\begin{enumerate}
		\item \[
		\dim_A^\theta(\Gamma_T)\leq
		\frac{2-\gamma-\theta}{1-\theta},\quad \theta\in(0,\gamma).
		\] 
		\item  If $\lambda^2N<1$, then we obtain a lower bound on the spectrum given by
		\[
		\dim_A^\theta(\Gamma_T)\geq 2 -\gamma+ \frac{\theta}{1-\theta}\frac{1-\gamma }{\gamma}\left(2\gamma-1\right),  \quad \theta\in(0,\gamma).
		\] 
		\item If $\lambda^2N<1$, then the Assouad dimension is bounded below by
		\[
		\dim_A(\Gamma_f)\geq 1 + \gamma.
		\]
	\end{enumerate}
\end{theorem}
We next derive exact estimates for the Assouad spectrum and quasi-Assouad dimension of the classical Takagi function.
\begin{example}
 	Let $T_{\lambda,N}^\phi$ denote the classical Takagi function where $\lambda N = 1$, that is,
 	\[
 	T_{\lambda,N}^\phi=\sum_{k=0}^{\infty}\frac{1}{N^k}\phi(N^k x).	\]
 	By applying Theorem \ref{Upperbound}, we get 
 	$\dim_A^{\theta}(\Gamma_T)=1$
 	for $\theta\in(0,1)$ and hence the quasi-Assouad dimension of the classical Takagi function is also equal to $1.$ 
 	\begin{figure}[H]
 		\centering
 		\begin{tikzpicture}
 			\begin{axis}[
 				width=15cm,
 				height=8cm,
 				xlabel={$x$},
 				ylabel={$W_{\lambda,N}^{\phi}(x)$},
 				grid=major,
 				enlarge x limits=false,
 				enlarge y limits=false,
 				line join=round,
 				]
 				
 				\addplot[
 				thick,
 				blue
 				]
 				table[x index=0,y index=1] {takagipoints2.dat};
 			\end{axis}
 		\end{tikzpicture}
 				\caption{
 				Takagi function with $N=2$.
 			}
 \end{figure}
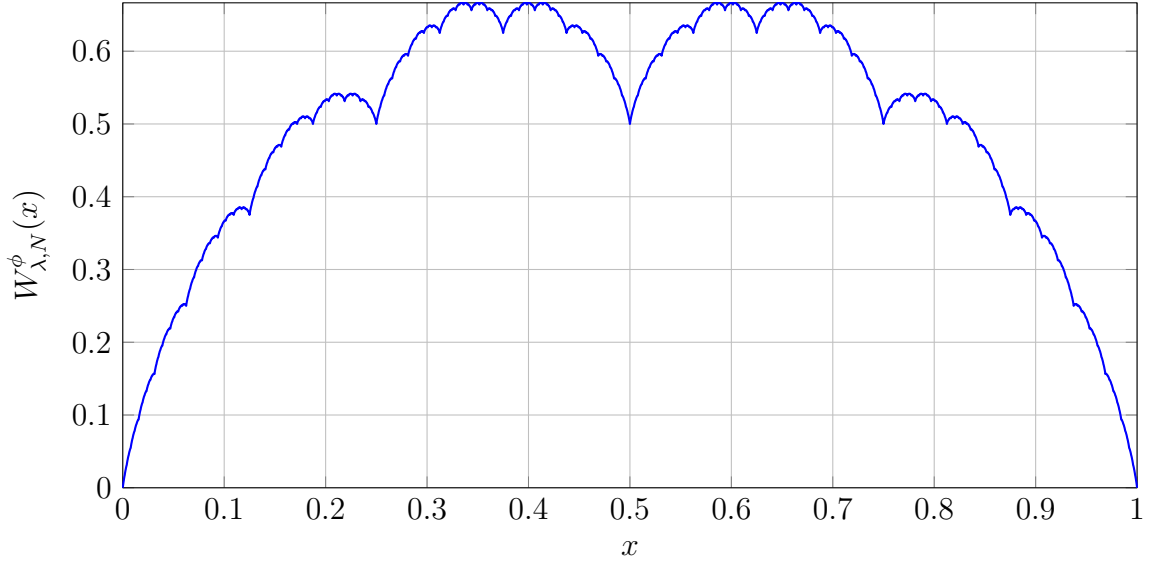
 			
\end{example}
The following example illustrates a non-trivial lower bound for the Assouad spectrum of the general Takagi function.
\begin{example}\label{ex-5.3}
	Let $T_{\lambda,N}^\phi$ denote the Takagi function where $\phi(x)=\text{Dist}(x,\mathbb{Z})$. Let $\lambda=0.6$ and $N=2$. Since $
	\log_N \frac{1}{\lambda} = 0.736\ldots>\frac{1}{2},
	$ the lower bound of the spectrum is non-trivial and we have
	\[
	\dim_A^\theta(\Gamma_g)\leq
	\begin{cases}
		\frac{1.274\ldots-\theta}{1-\theta}, & 0\leq\theta<0.736 \\[2mm]
		2, &0.736\leq\theta\leq1
	\end{cases} 
	\]
	and 
	\[\dim_A^\theta(\Gamma_f)\geq 1.274 + \frac{\theta}{1-\theta}\left(\frac{0.274}{0.736}\right)(0.472)
	\]
	for $\theta\in(0,0.736)$. 
	Therefore we have $
	\dim_A(\Gamma_f)\geq 1.746.
	$
	
	\begin{figure}[H]
		\centering
		
		\begin{minipage}{0.48\textwidth}
			\centering
			\begin{tikzpicture}
				\begin{axis}[
					width=\textwidth,
					height=6cm,
					xlabel={$x$},
					ylabel={$W_{\lambda,N}^{\phi}(x)$},
					grid=major,
					enlarge x limits=false,
					enlarge y limits=false,
					line join=round,
					]
					
					\addplot[
					thick,
					blue,
					each nth point=2
					]
					table[x index=0,y index=1] {takagipoints5.dat};
					
				\end{axis}
			\end{tikzpicture}
			\caption{
				General Takagi function with
				$\lambda=0.6$,
				$N=2$.
			}
		\end{minipage}
		\hfill
		\begin{minipage}{0.48\textwidth}
			\centering
			\begin{tikzpicture}
			\begin{axis}[
				width=\textwidth,
				height=6cm,
				axis lines=middle,
				xlabel={$\theta$},
				ylabel={$\dim_A^\theta(G)$},
				xmin=0, xmax=1,
				ymin=1, ymax=2.2,
				samples=200,
				clip=true,
				]
				
				\addplot[darkgray, thick, domain=0.0001:0.736] 
				{(1.264-x)/(1 - x)};
				
					\addplot[lightgray, thick, domain=0.0001:0.736] 
				{1.264 + (x/(1-x))*0.273 *0.472/0.736};
				\addplot[darkgray, thick, domain=0.736:1] 
				{2};
				\addplot[dashed, thin ,domain=1:2] ({0.736}, x);
				
			\end{axis}
		\end{tikzpicture}
		\caption{Bounds of the Assouad spectrum of the general Takagi function}
		\end{minipage}
		
	\end{figure}
	
\end{example}

\subsection{A general self-affine function} 
In 2025, Jiang and Ruan \cite{Ruan1} proved the following result regarding the box dimension of a generalized affine fractal interpolation function and the vertical scaling matrices.

\begin{theorem}\cite[Theorem 4.6]{Ruan1}\label{Ruan}
	Assume $\alpha$ is non-zero on $I$. Then \begin{equation*}
		\overline{\dim_B}\; \Gamma_f
		\leq
		\max \left\{
		1,
		1+\frac{\log \rho^{*}}{\log N}
		\right\},\; \, 
		\underline{{\dim_B}}\; \Gamma_f
		\geq
		\max \left\{
		1,
		1+\frac{\log \rho_{*}}{\log N}
		\right\}.
	\end{equation*}
\end{theorem}
It is important to note that previous examples have a constant vertical scaling function, i.e., $\alpha(x)\equiv\lambda$, while the main result can compute the bounds on the Assouad spectrum when $\alpha$ is Lipschitz. Here we present a general example, which also includes the calculation of box dimension using Theorem \ref{Ruan}
 \begin{example}\label{ex-affine}
     Let $I=[0,1], N=3$, $x_n=n/3$ for $n=0,1,2,3$. Let $\alpha$ be the vertical scaling function given by \[
     \alpha(x)=\frac{2}{5}+ \frac{\cos(2\pi x)}{4}, \quad x\in [0,1].\] Since $\alpha(x)>0$ for any $x\in[0,1]$, we have $\rho^*=\rho_*=:\rho_{\alpha}$.
    Let $\phi(x)=\cos(2\pi x)$ and define $W_n$ for $n=1,2,3$ by \[
 W_n(x,y)=\bigg(\frac{x+n-1}{N},\alpha(x) y+ \phi\bigg(\frac{x+n-1}{N}\bigg)\bigg)\;  \text{for all}\; x,y\in [0,1].\] Let  $y_0=y_3=\frac{20}{7}$ and $y_1=y_2=\frac{19}{14}$. Then we have \[
 W_n(x_0,y_0)=(x_{n-1},y_{n-1}) \;  \text{and} \;
 W_n(x_3,y_3)=(x_{n},y_{n}). \]
 \begin{figure}[H]
 	\centering
 	
 	\begin{minipage}{0.48\textwidth}
 		\centering
 		 \begin{tikzpicture}
 			\begin{axis}[
 				width=\textwidth,
 				height=6cm,
 				xmin=0,
 				xmax=1,
 				ymin=-1.5,
 				ymax=3,
 				axis lines=left
 				]
 				
 				\addplot[
 				blue,
 				thick,
 				no markers
 				]
 				table {generalized_afif5.dat};
 				
 			\end{axis}
 		\end{tikzpicture}
 		\caption{
 			Generalized affine function with $\alpha(x)=\frac{2}{5}+\frac{cos(2\pi x)}{4}.$
 		}
 	\end{minipage}
 	\hfill
 	\begin{minipage}{0.48\textwidth}
 		\centering
 		\begin{tikzpicture}
 			\begin{axis}[
 				width=\textwidth,
 				height=6cm,
 				axis lines=middle,
 				xlabel={$\theta$},
 				ylabel={$\dim_A^\theta(\Gamma_G)$},
 				xmin=0, xmax=1,
 				ymin=1, ymax=2.2,
 				samples=200,
 				clip=true,
 				]
 				
 				\addplot[darkgray, thick, domain=0.0001:0.5783] 
 				{(1.166-0.558*x)/(1 - x)};
 					\addplot[lightgray, thick, domain=0.0001:0.392] 
 				{1.166+x/(1-x)*0.602*0.226/0.392};
 				
 				\addplot[darkgray, thick, domain=0.5783:1] 
 				{2};
 				\addplot[dashed, thin ,domain=1:1.392] ({0.392}, x);
 				\addplot[dashed, thin ,domain=1:2] ({0.5783}, x);
 				
 			\end{axis}
 		\end{tikzpicture}
 		\caption{Assouad spectrum of the generalized affine function.}
 	\end{minipage}
 	
 \end{figure}
 Therefore, $\{W_n\}_{n=1}^3$ defines a generalized affine FIF. Now we approximate the value of $\rho_{\alpha}$ using the vertical scaling matrices. Here are the matrices $\overline{P}_1$ and $\overline{P}_2$ for the function $\alpha$:
\[
\overline{P}_1 =
\begin{pmatrix}
	0.6500 & 0.5915 & 0.4434 \\
	0.2750 & 0.1651 & 0.2750 \\
	0.4434 & 0.5915 & 0.6500
\end{pmatrix}.
\]
 
\[
\overline{P}_2 =
\begin{pmatrix}
	0.6500 & 0.6433 & 0.6234 & 0 & 0 & 0 & 0 & 0 & 0 \\
	0 & 0 & 0 & 0.5915 & 0.5493 & 0.4990 & 0 & 0 & 0 \\
	0 & 0 & 0 & 0 & 0 & 0 & 0.4434 & 0.3855 & 0.3283 \\
	0.2750 & 0.2284 & 0.1911 & 0 & 0 & 0 & 0 & 0 & 0 \\
	0 & 0 & 0 & 0.1651 & 0.1517 & 0.1651 & 0 & 0 & 0 \\
	0 & 0 & 0 & 0 & 0 & 0 & 0.1911 & 0.2284 & 0.2750 \\
	0.3283 & 0.3855 & 0.4434 & 0 & 0 & 0 & 0 & 0 & 0 \\
	0 & 0 & 0 & 0.4990 & 0.5493 & 0.5915 & 0 & 0 & 0 \\
	0 & 0 & 0 & 0 & 0 & 0 & 0.6234 & 0.6433 & 0.6500
\end{pmatrix}.
\]

From the matrices we obtain $\rho_{\alpha}\approx 1.2$, and therefore from by applying Theorem \ref{Ruan}, we get:
\[
\dim_B(\Gamma_f)=1 + \log_N\rho_{\alpha}\approx 1.1659\ldots
\]
\begin{table}[ht]
	\centering
	\begin{tabular}{cccccccc}
		\toprule
		$\mathbf{k}$ & 1 & 2 & 3 & 4 & 5 & 6 & 7 \\
		\midrule
		$\rho(\overline{P}_k)$
		& 1.3646 & 1.2553 & 1.2183 & 1.2061 & 1.2020 & 1.2006 & 1.2002 \\
		
		$\rho(\underline{P}_k)$
		& 1.0324 & 1.1447 & 1.1816 & 1.1938 & 1.1979 & 1.1993 & 1.1998 \\
		\bottomrule
	\end{tabular}
	\caption{Spectral radii of the matrices $\overline{P}_k$ and $\underline{P}_k$ for different values of $k$.}
	\label{tab:spectral_radii}
\end{table}
 Now $||\alpha||_\infty=\frac{2}{5}+\frac{1}{4}=\frac{13}{20}$. Hence we have the following upper bound
	\[
 \dim_A^\theta(\Gamma_f)\leq
 1.166+\frac{\theta}{1-\theta}\left( 1+\log_3\frac{13}{20} \right)=
 \frac{1.166-0.558\theta}{1-\theta}
 \]
 for $\theta\in(0,0.5783\ldots)$. Since $ ||\alpha||_\infty=\frac{13}{20}<0.834=N^{1-\dim_B(\Gamma_f)}$, the spectrum has a non-trivial lower bound and hence by applying  Theorem \ref{Lowerbound}, we obtain
 \[
  \dim_A^\theta(\Gamma_f)\geq 1.166+ \frac{\theta}{1-\theta}\left(\frac{0.601}{0.392}\right)0.226,
 \; \theta\in(0,0.392\ldots)\] and hence $
  \dim_A(\Gamma_f)\geq 1.392\ldots.
 $
 \end{example}
 \begin{remark}
 Note that in Example \ref{ex-affine}, the values of \(\theta\) for which the upper bound is less than $2$ are not the same as the values of $\theta$ for which the lower bound is less than $2$. This difference occurs because the vertical scaling function is not constant. In Examples \ref{ex-5.1} and \ref{ex-5.3}, the vertical scaling functions were constant, so the corresponding ranges of $\theta$ were the same.
 \end{remark}
\section*{Acknowledgments}
	This work is supported by CSIR-HRDG ASPIRE scheme (grant no-25WS(014)/2023-24/EMR-II/ASPIRE). The first and second authors are thankful to CSIR for the funding. 
\bibliographystyle{amsplain}
	\bibliography{RefAssouad}
\end{document}